\documentclass[reqno]{amsart}
\usepackage{latexsym,amssymb}
\usepackage{upref}

\newtheorem{theorem}{Theorem}[section]
\newtheorem{lemma}[theorem]{Lemma}
\newtheorem{corollary}[theorem]{Corollary}

\newtheorem{conjecture}[theorem]{Conjecture}
\newtheorem{proposition}[theorem]{Proposition}
\theoremstyle{definition}

\newtheorem{remark}[theorem]{Remark}
\numberwithin{equation}{section}
\newcommand{\abs}[1]{\lvert#1\rvert}

\begin{document}

\title[{A real part theorem for  derivatives of analytic functions}]{A real part theorem
for the higher derivatives of analytic functions in the unit disk}


\author{David Kalaj}
\address{University of Montenegro, Faculty of Natural Sciences and
Mathematics, Cetinjski put b.b. 81000 Podgorica, Montenegro}
\email{davidk@ac.me}

\author{Noam D. Elkies}
\address{Department of Mathematics, Harvard University,
Cambridge, MA 02138}
            \email{elkies@math.harvard.edu}


\begin{abstract}  Let $n$ be a positive integer. Let
$\mathbf U$ be the unit disk, $p\ge 1$ and let $h^p(\mathbf U)$ be
the Hardy space of harmonic functions. Kresin  and  Maz'ya in a
recent paper found a representation for the function $H_{n,p}(z)$
in the inequality
$$|f^{(n)} (z)|\leq H_{n,p}(z)\|\Re(f-\mathcal P_l)\|_{h^p(\mathbf U)}, \ \ \ \Re f\in h^p(\mathbf U), z\in \mathbf
U,$$ where $\mathcal P_l$ is a polynomial of degree $l\le n-1$. We
find or represent the sharp constant $C_{p,n}$ in the inequality
$H_{n,p}(z)\le \frac{C_{p,n}}{(1-|z|^2)^{1/p+n}}$. This extends a
recent result of Kalaj  and Markovi\'c, where only the case $n=1$ was considered. As a corollary, an inequality for the modulus of
 $n-{th}$ derivative of an analytic function defined in a complex
domain with the bounded real part is obtained. This result improves
a recent result of Kresin and Maz'ya.
\end{abstract}
\maketitle
\section{Introduction and  statement of the results}
 A harmonic function $f$ defined in the unit disk $\mathbf U$ of the
complex plane $\mathbf C$ belongs to the harmonic Hardy class
$h^p=h^p( \mathbf U)$, $1\leq p<\infty$ if the following growth
condition is satisfied
\begin{equation}\label{ee}\Vert f\Vert_{h^p}:
=\left(\sup_{0<r<1}\int_{\mathbf T} |f(re^{it})|^p dt
\right)^{1/p}<\infty\end{equation} where $\mathbf T$ is the unit
circle in the complex plane $\mathbf C$. The space $h^{\infty}(
\mathbf U)$ { consists of all} bounded harmonic functions.

If $f\in h^p(\mathbf U)$, then there exists the finite radial limit
$$\lim_{r\to 1^-}f(r\zeta) = f^*(\zeta)\ (\text{a.e. on } \mathbf T
)$$ and the boundary function $f^*$ belong to the space $L^p(\mathbf
T)$ of $p$-integrable functions on the circle.

It is well known that {a harmonic function $f$} in the Hardy class
$h^p(\mathbf U)$ can be represented as { the} Poisson integral
$$f(z)=\int_{\mathbf T}P(z,\zeta)d\mu(\zeta), z \in  \mathbf U$$
where $$P(z,\zeta)=\frac{1-|z|^2}{ |z-\zeta|^2}, z\in \mathbf U,
\zeta\in \mathbf T$$ is {the} Poisson kernel and $\mu$ is { a}
complex Borel measure. In the case $p>1$ this measure is absolutely
continuous with respect to the Lebesgue measure  and
$d\mu(\zeta)=f^*(\zeta)d\sigma(\zeta)$. Here $d\sigma$ is Lebesgue probability measure in the unit circle. Moreover
we have
\begin{equation}\label{more}\Vert f\Vert_{h^p}=\|f^*\|_p,\ \ \ p>1\end{equation} and $$\Vert f\Vert_{h^1}=\|\mu\|$$
where we denote by $\|\mu\|$ { the} total variation of the measure
$\mu$.

For previous facts we refer to the book  \cite[Chapter~6]{ABR}.  In
the sequel for $p\ge 1$ and $m$ a positive integer, as in
\cite{real} we use the notation
$$E_{m,p}(\Re f) := \inf _{ \mathcal P\in \mathfrak P_m} \|\Re(f -
\mathcal P)\|_{h^p}$$ for the best approximation of $\Re f$ by the
real part of algebraic polynomials in the $h^p(\mathbf U)$-norm, where
$\mathfrak P_m$ is the set of  all  algebraic polynomials of degree
at most $m$.

The starting position of this paper is the following proposition of Maz'ya and Kresin \cite[Proposition~5.1]{real}.
\begin{proposition}
 Let $f$ be analytic on $\mathbf U$  with $\Re f\in h^p(\mathbf U)$, $1\le p\le
\infty$. Further, let $n\ge 1$, and let $\mathcal P_l$ be a
polynomial of degree $l\le n-1$. Then for any fixed point $z$, $|z|
= r < 1$, the inequality \begin{equation}\label{inin}
|f^{(n)}(z)|\le H_{n,p}(r)\|\Re(f-\mathcal P_l)\|_{h^p}
\end{equation}
holds with the sharp factor   \begin{equation}\label{izi}
H_{n,p}(r)=\frac{n!}{\pi}\sup_\alpha\left\{\int_{|\zeta|=1}\left|\Re
\frac{\zeta e^{i\alpha}}{(\zeta-r)^{n+1}}\right|^q
|d\zeta|\right\}^{1/q}\end{equation} and $1/q+1/p=1$. In particular
$$|f^{(n)}(z)|\le  H_{n,p}(r) E_{n-1,p}(\Re f).$$

\end{proposition}

For $p=2$ ($q=2$) and $p=1$ ($q=\infty$) the function $H_{n,p}(r)$
has been calculated explicitly in \cite{real}. We refer to
\cite{real} for the connection of \eqref{inin} and the famous
Hadamard-Borel-Carath\'eodory inequality: $$|f(z)-f(0)|\le
\frac{2}{1-|z|^2} \sup_{|\zeta|<1}\Re[f(\zeta)-f(0)].$$ The aim of
this paper is to obtain some explicit estimations of $H_{n,p}(r)$
for general $p$. The results of this paper are
\begin{theorem}[Main theorem]\label{condi}
Let $1\le p\le \infty $ and let $q$ be its conjugate.  Let $f$ be
analytic on the unit disk $\mathbf U$ with $\Re f\in h^p(\mathbf
U)$, $1\le p\le \infty$. Further, let $n\ge 1$, and let $\mathcal
P_l$ be a polynomial of degree $l\le n-1$.  We have the following
sharp inequality
\begin{equation}\label{dl}|f^{(n)}(z)|\le C_{p,n}
(1-r^2)^{-1/p-n}\|\Re (f-\mathcal P_l)\|_{h^p},\end{equation} where
\begin{equation}\label{newap}C_{p,n}=\frac{n!}{\pi}2^{n+1-1/q}\max_{0\le \beta\le
\pi/2}F^{1/q}_q(\beta)\end{equation} and
\begin{equation}\label{fbeta}F_q(\beta)=\int_0^{\pi }\left|\sin^{{(n+1)-2/q}}
v\cos[v(n+1) +\beta - \frac{\pi}{2}(n-1)]\right|^q dv.\end{equation} In particular
$$|f^{(n)}(z)|\le C_{p,n} (1-r^2)^{-1/p-n}E_{n-1,p}(\Re f).$$
\end{theorem}

\begin{remark} In connection with Theorem~\ref{condi}, we conjecture that (c.f. Conjecture~\ref{cone}) $$\max_{0\le \beta\le
\pi/2}F_q(\beta)=\max\{F_q(0),F_q(\pi/2)\}.$$ We have the solution
for $q=1$ presented in Theorem~\ref{tete}. We list some known partial solutions.
\begin{itemize}
    \item The Hilbert case (see \cite[Eq.~5.5.3]{real} or \cite[Eq.~(6.1.4)$_2$]{maro}): for
$q=2$ and all $n$, the corresponding function is
$$F_q(\beta)=\frac{2^n}{\pi^{3/4}}\sqrt{\frac{ \Gamma[1/2+n]}{
    \Gamma[1+n]}}.$$
    \item For $q=\infty$ and all $n$ (\cite[Eq.~5.4.2]{real} for
$\gamma=1$), $F_q(\beta)=2^{n+1}$.
    \item For $n=1$ and all $q$, (see \cite{mada}) we have $$\max_{0\le \beta\le
\pi/2}F_q(\beta)=\left\{
                   \begin{array}{ll}
                     F_q(0), & \hbox{if $q>2$,} \\
                     F_q(\pi/2), & \hbox{if $q\le 2$}
                   \end{array}
                 \right..$$
   \end{itemize}
We also refer to related sharp inequalities for the derivatives of
analytic functions defined in the unit disk \cite{rus}.
\end{remark}
\begin{theorem}\label{tete} Let $f$ be
analytic on the unit disk $\mathbf U$ with bounded real part $\Re f$ and assume that
 $\beta\in[0,\pi]$. Then \begin{equation}\label{dl1}|f^{(n)}(z)|\le C_{n}
(1-|z|^2)^{-n}\mathcal O_{n,\Re f}(\mathbf U),\end{equation}
where
$$C_{n}=\left\{
          \begin{array}{ll}
            \frac{1}{n\pi}\frac{((2{m})!)^2}{(m!)^2}, & \hbox{if $n=2m-1$;} \\
            \frac{n!}{\pi}\max\{F(\beta):0\le\beta\le\pi\}& \hbox{if $n=2m$}
          \end{array}
        \right.,$$

$$F(\beta)=\frac{2}{n}\sum_{k=1}^{n+1}\sin^{n+1}\frac{k\pi-\beta}{n+1}$$ and
$$\mathcal O_{n,\Re f}(\mathbf U)=\inf_{\mathcal P\in \mathfrak
P_{n-1}}\mathcal O_{\Re (f-P)}(\mathbf U)$$ and $\mathcal O_{\Re
f}(\mathbf U)$ is the oscillation of $\Re f$ on the unit disk
$\mathbf U$.
\end{theorem}

\begin{remark}
If $w=\Re f$ is a real harmonic function, where $f$ is an analytic
function defined on the unit disk, then the Bloch constant of $w$ is
defined by
$$\beta_w=\sup_{z\in\mathbf U}(1-|z|^2)| \nabla
w(z)|=\sup_{z\in\mathbf U}(1-|z|^2)| f'(z)|$$ and is less than or equal
to $C_1=4/\pi$ provided that the oscillation of $w$  in the unit
disk is  $\le 1$. This particular case is well known in the
literature see e.g. \cite{cor, kha, kavu}. In a similar manner we
define the Bloch constant of order $n$ of a harmonic function $w=\Re
f$:
$$\beta_{n,w}=\sup_{z\in\mathbf U}(1-|z|^2)^n|f^{(n)}(z)|$$ and by the
previous corollary we find out that $\beta_{n,w}\le C_n$ provided
$n$ is an odd integer and the oscillation $\mathcal O_{n,\Re
f}(\mathbf U)$ is at most $1$.
\end{remark}

The following theorem improves one of the main results in
\cite{makr} (see \cite[Corollary~7.1]{makr}).
\begin{corollary}\label{ariora}
Let $\Omega$ be a subdomain of $\mathbf C$. Let $z\in\Omega$, and assume that $a_z\in\partial\Omega$ such that $|z-a_z|=d_z =
\mathrm{dist} (z ,\partial \Omega)$ and that $[\zeta,a_z]$ is the maximal interval containing $z$ with $d_\zeta=|\zeta-a_z|$.  Let $f$ be a holomorphic
function in $\Omega$ with its real part in Lebesgue space
$L^\infty(\Omega)$ and let $\| \mathrm{Re} f\|_{L^\infty(\Omega)}\le
1.$ Then the inequality
\begin{equation}\label{better}d_z^{n}
|f^{(n)}(z)| \le\frac{C_n}{(2-d_z/d_\zeta)^n}, z\in \Omega\end{equation} holds with $C_n:=C_{\infty,n}$
defined in \eqref{newap}. In particular,
\begin{equation}\label{bebe} d_z^{{2m-1}} |f^{(2m-1)}(z)| \le
\frac{1}{n\pi}\frac{((2{m})!)^2}{(m!)^2}\frac{1}{(2-d_z/d_\zeta)^n}, \ \ z\in
\Omega.\end{equation}
\end{corollary}

\begin{proof} Let $z\in\Omega$ and assume that  $\zeta\in \Omega$ satisfies the condition of the theorem. Then $D_\zeta:=\{w:|w-\zeta|<d_\zeta\}\subset \Omega$.
Define $g(w)=f(\zeta+wd_\zeta)$, $w\in\mathbf{U}$. Then $\mathrm{Re}\, g\in h^\infty(\mathbf
U)$ and $g^{(n)}(w)=d_\zeta^n f^{(n)}(\zeta+wd_\zeta)$. By the maximum principle we
have
$$\| \mathrm{Re}\, f\|_{h^\infty(D_\zeta)}\le \| \mathrm{Re}
f\|_{L^\infty(\Omega)}.$$ By applying
Theorem~\ref{condi} and Theorem~\ref{tete} to $g$ we have $$(1-|w|^2)^nd_\zeta^n |f^{(n)}(\zeta+wd_\zeta)|\le C_n.$$  As $z\in[\zeta,a_z]$ it follows that $z= \zeta+s(a_z-\zeta)=\zeta+w d_\zeta$, where $w=se^{i\phi}\in\mathbf{U}$. Since $d_\zeta=(1-s)^{-1}d_z$, and $|w|=|(z-\zeta)/d_\zeta|=s=(d_\zeta-d_z)/d_z$ we obtain
that
$$(1-s^2)^n(1-s)^{-n}d^n_z |f^{(n)}(z)|\le C_n,$$ and $$d^n_z |f^{(n)}(z)|\le C_n \frac{1}{(1+s)^n}=C_n\frac{1}{(2-d_z/d_\zeta)^n}.$$
\end{proof}

\begin{remark}
In \cite[Corollary~7.1]{makr} Kresin and Maz'ya proved that
\begin{equation}\label{bebepo} \lim_{\epsilon\to 0^+}\sup_{z:d_z=\epsilon}d_z^{{2m-1}} |f^{(2m-1)}(z)| \le
 2^{-n} C_{n},\end{equation} under the condition that $\Omega$ is a
planar domain with certain smoothness condition on the boundary, namely assuming that there is $r>0$ such that for $a\in\partial \Omega$, there is a disk $D_a\subset \Omega$ of radius $r$ with $a\in\overline{D_a}$. Then $d_\zeta\ge r$ for $z\in\Omega$. The inequality \eqref{bebepo} follows by using \eqref{better} and letting $\epsilon\to 0$.

\end{remark}

\section{Proof of Theorem~\ref{condi}}
In view of \eqref{izi}, we deal with the function
\begin{equation}\label{maz}I_\alpha(r)=\int_0^{2\pi}\left|\Re\frac{e^{i(\alpha+t)}}{(r-e^{it})^{n+1}}\right|^qdt,\ \ \ 0\le r<1.\end{equation}
By making use of the change
$$ e^{it}
=\frac{r-e^{is}}{1-re^{is}},$$ we obtain $$dt =
\frac{1-r^2}{|1-re^{is}|^2}ds$$ and  $$r-e^{it}
=\frac{(1-r^2)e^{is}}{1-re^{is}}.$$ We arrive at the integral
\[\begin{split}I_\alpha(r)&=\int_0^{2\pi}\left(1-r^2\right)^{1-q-n q} \left(1+r^2-2
r \cos s\right)^{-1+q} \left|\Re\left[e^{i (\alpha+s)} \left(e^{i
s}-r\right)^{-1+n}\right]\right|^qd s\\&= \left(1-r^2\right)^{1-q-n
q} \int_0^{2\pi} f_\alpha(r,e^{is}) ds\end{split}\] where
$$f_\alpha(z,e^{is})=|\Re[e^{i(\alpha+s)}(z-e^{is})^{n-1}]|^q{|z-e^{is}|^{2q-2}}.$$
In order to continue, let's prove first two lemmas.
\begin{lemma}
$f_\alpha$ is subharmonic in $z$.
\end{lemma}
\begin{proof}
We refer to \cite[Chapter~4]{lib} and
\cite[Chapter~I~\S~6]{gar} for some basic properties of subharmonic
functions.
Recall that a continuous function $g$ defined on a region $G\subset
\mathbf  C$ is subharmonic if for all $w_0 \in G$ there exists
$\varepsilon
> 0$ such that \begin{equation}\label{qwe}g(w_0) \le \frac{1}{2\pi}\int_0^{2\pi}  g(w_0 + re^{it} ) dt, \ \ 0 < r
<\varepsilon.\end{equation}  If $g(w_0)=0$,  since $g$ is non-negative, then
\eqref{qwe} holds. If $g(w_0)> 0$, then there exists a neighborhood
$U$ of $w_0$ such that $g$ is of class $C^2 (U)$ and $g(w)>0$ $w\in
U$. Thus if $g$ is $C^2$ where it is positive, then it is enough to check that the Laplacian is non-negative there.

Let $w=e^{i\frac{\alpha+s}{n-1}} z-e^{is-i\frac{(\alpha+s)}{n-1}}$
and define
$$g(w):=f_\alpha(z,e^{is})=|\Re[w^{n-1}]|^q{|w|^{2q-2}}.$$
Assume that $\Re(w^{n-1})>0$. Then $$g_w=2^{-q} \bar w (\bar w
w)^{q-4} (\bar w^{n-1} +
   w^{n-1})^{q-1} [(-1 + q) \bar w^n w + (-1 + n q) \bar w w^n].$$
Further \[\begin{split}g_{w\bar w}&=
   \frac{(q-1) (\bar w w)^{q-4} (\bar w^{n-1}+w^{n-1})^{q-2}}{2^{q+2}}\\&\times
   ((-1+n q) \bar w^{2 n} w^2+(-1+n q) \bar w^2 w^{2 n}+(-2+q+n^2 q) \bar w^{1+n} w^{1+n})
   .\end{split}\]
Observe next that $$\frac{(q-1) (\bar w w)^{q-4} (\bar
w^{n-1}+w^{n-1})^{q-2}}{2^{q+2}}=\frac{ (q-1) |w|^{2q-8} (\Re
w^{n-1})^{q-2}}{16}\ge 0$$ and \[\begin{split}(-1&+n q) \bar w^{2 n}
w^2+(-1+n q) \bar w^2 w^{2 n}+(-2+q+n^2 q) \bar w^{1+n}
w^{1+n}\\&=(-2+q+n^2 q)|w|^{2n+2}-2(-1+nq)\Re[\bar w^2 w^{2n}]\\&\ge
q(n-1)^2|w|^{2n +2}\ge 0.\end{split}\] Similarly we treat the case
$\Re(w^{n-1})<0$. Therefore $\Delta g=4g_{w\bar w}\ge 0$ for
$\Re(w^{n-1})\neq 0$. This implies that $g$ is subharmonic in the
whole of $\mathbf C$. Since $$f_\alpha(z, e^{is})
)=g(e^{i\frac{\alpha+s}{n-1}} z-e^{is-i\frac{(\alpha+s)}{n-1}}),$$
we have that $\Delta f_\alpha(z, e^{is})=\Delta g(az+b)$ which
implies that $z\to f_\alpha(z, e^{is})$ is subharmonic.
\end{proof}

\begin{lemma}\label{mm} For $\alpha\in[0,\pi]$ we have \[\begin{split}I_\alpha(r)&\le  \left(1-r^2\right)^{1-q-n
q}\max_{0\le t\le 2\pi} \int_0^{2\pi} f_\alpha(e^{it},e^{is}) ds\\&=
\left(1-r^2\right)^{1-q-n q}\int_0^{2\pi} f_\beta(1,e^{is})
ds,\end{split}\] for some $\beta$ possibly different from $\alpha$.
\end{lemma}
\begin{proof}
Since $z\to f_\alpha(z,e^{it})$ is subharmonic
in $|z|< 1$, and $(t,z)\to f_\alpha(z,e^{it})$ is continuous in $[0,\pi]\times\mathbf{U}$, then the integral mean $I(z)=\int_0^{2\pi}
f_\alpha(z,e^{is})ds$ is a subharmonic function in $|z|< 1$.
Therefore
$$\int_0^{2\pi} f_\alpha(z,e^{is}) ds\le \max_t\int_0^{2\pi}
f_\alpha(e^{it},e^{is}) ds.$$ Since
\[\begin{split}f_\alpha(e^{it},e^{is})&=|\Re[e^{i(\alpha+s)}(e^{it}-e^{is})^{n-1}]|^q{|e^{it}-e^{is}|^{2q-2}}\\&=
|\Re[e^{i(\beta+u)}(1-e^{iu})^{n-1}]|^q{|1-e^{iu}|^{2q-2}},\end{split}\]
for $u=s-t$ and $\beta=\alpha+nt$ we obtain the second statement of
the lemma.
\end{proof}

\begin{proof}[Proof of Theorem~\ref{condi}]
As above, we have
\[\begin{split}|\Re[e^{i(\alpha+s)}&(e^{it}-e^{is})^{n-1}]|^q\\&=
2^{(n-1)q/2} \left|\cos[\beta-\frac{\pi}{2}(n-1) +u
\frac{n+1}{2}]\right|^q \left(1-\cos
u\right)^{(n-1)q/2}\end{split}\] and
$${|1-e^{iu}|^{2q-2}}=2^{q-1}\left(1-\cos u\right)^{q-1}.$$ In view of Lemma~\ref{mm} we have
\[\begin{split}I_\alpha(r)
 =2^{(n+1)q/2-1} (1 - r^2)^{1 - (1 + n)
q}F_q(\beta),\end{split}\] where
\begin{equation}\label{dolar}F_q(\beta)=\int_0^{2\pi }(1-\cos u
)^{\frac{(n+1)q-2}{2}}\abs{\cos[u\frac{n+1}{2} +\beta -
\frac{\pi}{2}(n-1)]}^qdu.\end{equation} Moreover
$$F_q(\beta)=2^{(n+1)q/2}\int_0^{\pi
}|\phi_\beta(v)|^q dv $$ where \begin{equation}\label{fb}\phi_\beta(v)=\sin^{{(n+1)-2/q}}
v\cos[v(n+1) +\beta - \frac{\pi}{2}(n-1)].\end{equation} It can be proved
easily that $F_q(\beta)=F_q(\pi-\beta)$. The last fact
implies that it is enough to find the maximum in $[0,\pi/2]$.
\end{proof}

\section{The case $q=1$ and the proof of Theorem~\ref{tete}}
We divide the proof into two cases and use the notation $F=F_q$.
\subsection{The odd $n$.}  For $n=2m-1$ and $q=1$
we have \begin{equation}\label{lm}F(\beta)= 2^{m}\int_0^{\pi
}\sin^{n-1} v \abs{\cos[(n+1)v +\beta]}dv.
\end{equation}Then
$$\frac{d}{dx}\frac{\cos(\beta + n x) \sin^{n}
x}{n}=(-1)^{m-1}\phi_\beta(x)=\sin^{n-1} x\cos[\beta+(n+1)x ].$$ Since  $F$ is
$\pi-$periodic we can assume that $-\pi/2\le \beta\le \pi/2$. Assume
that $0\le \beta<{\pi}/{2}$ (the second case can be treated
similarly). Then $\cos[\beta+2mx ]\ge 0$ and $0\le x\le\pi$ if and only if one of the
following relations hold
\begin{itemize}
    \item $\beta\le \beta+2mx<\frac{\pi}{2}$
    \item $-\frac{\pi}{2}+2k\pi <\beta+2mx< \frac{\pi}{2}+2k\pi$, \
\ for $1\le k\le m-1$ or
    \item $-\frac{\pi}{2}+2m\pi <\beta+2mx< \beta+2m\pi$
\end{itemize}
 or, what is the same, if:
\begin{itemize}
    \item $a_0=0< x<b_0=\frac{{\pi}-2\beta}{4m}$
    \item $a_k:=\frac{-\pi +4 k
\pi -2 \beta }{4 m}<x<b_k:=\frac{\pi +4 k \pi -2 \beta }{4 m}$, \ \
for $1\le k\le m-1$ or
    \item $a_m:=\frac{-\pi +4 m \pi -2 \beta }{4
m}<x<b_m:=\pi .$
\end{itemize}
From \eqref{lm} for
$$g_n(x)=\frac{\cos(\beta + n x) \sin^{n} x}{n},$$
because $g_n(\pi)-g_n(0)=0$, we have
\[\begin{split}F(\beta)&=2^{m}\int_0^{\pi }|\phi_\beta(v)|dv
\\&=2\cdot 2^{m}\int_{0\le v\le \pi :\phi_\beta(v)\ge
0}\phi_\beta(v)dv -2^{m}\int_0^\pi \phi_\beta(v)dv\\&=
2^{m+1}\int_{0\le v\le \pi :\phi_\beta(v)\ge 0}\phi_\beta(v)dv.
\end{split}\]  Therefore
$$\frac{F(\beta)}{{2^{m+1}}}=\sum_{k=0}^{m}
[g_n(b_k)-g_n(a_k)]=g_n(b_0)-g_n(a_m)+\sum_{k=1}^{m-1}
[g_n(b_k)-g_n(a_k)].$$ But for $1\le k\le m$ $$g_n(b_k)=\frac{
\sin^{2m} b_k}{n} \ \ \text{and}\ \ \ g_n(a_k)=-\frac{ \sin^{2m}
a_k}{n}.$$ Therefore
\[\begin{split}F(\beta)=\frac{2^{m+1}}{n}\sum _{k=1}^{2m} \sin^{2m}\left[\frac{-2\beta+(2k-1) \pi }{4
m}\right]=\frac{2^{m+1}}{n}\sum _{k=1}^{2m}
\sin^{2m}\left[\frac{\gamma+k \pi }{2 m}\right],\end{split}\] where
$\gamma=-\pi/2-\beta$. Now by invoking \cite[Lemma~3.5]{mary}, we
have
\begin{equation}\label{fn}f(\beta):=\sum _{k=1}^{2m}
\sin^{2m}\left[\frac{\gamma+k \pi }{2
m}\right]=\frac{2}{B(\frac{1}{2},m)},\end{equation} and therefore
$$F(\beta)=\frac{4m}{n2^{m}}\binom{2m}{m}.$$
\subsection{The even $n$} For $n=2m$ and $q=1$
\begin{equation}\label{im} F(\beta)= 2^{(2m+1)/2}\int_0^{\pi }\sin^{{n-1}} v\abs{\sin[v(n+1)
+\beta]}dv.
\end{equation} Let $$g(x)=\frac{\sin(\beta + n x) \sin^{n}
x}{n}.$$ Since $$\frac{d}{dx}g(x)=(-1)^{m-1}\phi_\beta(x)=\sin^{{n-1}}
x\sin[\beta+(n+1)x ],$$ from \eqref{im} and
\[\begin{split}F(\beta)&=2^{(n+1)/2}\int_0^{\pi }|\phi_\beta(v)|dv
\\&=2\cdot 2^{(n+1)/2}\int_{0\le v\le \pi :\phi_\beta(v)\ge
0}\phi_\beta(v)dv -2^{(n+1)/2}\int_0^\pi \phi_\beta(v)dv\\&=
2^{(n+1)/2+1}\int_{0\le v\le \pi :\phi_\beta(v)\ge
0}\phi_\beta(v)dv,
\end{split}\] we obtain
\[\begin{split}\frac{F(\beta)}{2^{(n+1)/2+1}}&=g\left(\frac{-\beta+\pi }{1+n}\right)+\sum
_{k=1}^m \bigg(g\left(\frac{-\beta+(1+2 k) \pi
}{1+n}\right)-g\left(\frac{-\beta+2 k \pi
}{1+n}\right)\bigg).\end{split}\]
%
%
%
%
%
%
%
%
%
%
After some elementary transformations we obtain
\begin{equation}\label{split}\begin{split}2^{-(n+1)/2}mF(\beta)&=\sum _{k=1}^{n+1}
\sin^{1+n}\left[\frac{-\beta+k \pi
}{1+n}\right].\end{split}\end{equation} This finishes the proof of
Theorem~\ref{tete}.
\section{Appendix}
In this section we include a possible strategy how to determine the
maximum of the function $F$ in $[0,\pi/2]$ provided that $n=2m$ is
an even integer. First of all
$$2^{-(2m+1)/2} mF'(\beta)=-\sum_{k=1}^{1+2 m} \cos\left[\frac{-\beta+k \pi }{1+2
m}\right] \sin^{2 m}\left[\frac{-\beta+k \pi }{1+2 m}\right].$$ Let
$$h_k(\beta)=\cos\left[\frac{-\beta+k \pi }{1+2
m}\right] \sin^{2 m}\left[\frac{-\beta+k \pi }{1+2 m}\right].$$ Then
for $1\le k\le 2m$, $h_k(0)+h_{2m+1-k}(0)=0$, $h_{2m+1}(0)=0$  and
$h_k(\pi/2)+h_{2m+2-k}(\pi/2)=0$ and $h_{m+1}(\pi/2)=0$. It follows
that
$$F'(0)=F'\left(\frac{\pi}{2}\right)=0.$$ Thus $0$ and $\pi/2$ are
stationary points of $F$.


It can be shown that for $\gamma=\beta +\pi/2$ and for  $\pi/2\le
\gamma\le \pi$

\begin{equation}\label{salt}\frac{mF(\beta)}{2^{(n+1)/2}}= \sum _{j=0}^{ m}
\frac{(-1)^{j}}{2^{n}}\binom{1+n}{m-j}\frac{ \cos\left[\frac{ (1+2j)
\gamma}{1+n}\right]}{\sin\left[\frac{(1+2 j) \pi }{2
(1+n)}\right]}+2 \sin^{1+n}\left[\frac{\gamma-\pi/2
}{1+n}\right].\end{equation} We expect that the formula \eqref{salt}
can be more useful than \eqref{split} in finding the maximum of the
function $F(\beta)$, however it seems that the corresponding problem
is hard. By using the software "Mathematica 8"  we can see that
$F(0)<F(\beta)< F(\pi/2)$ provided that $n=4k$ and $0<\beta<\pi/2$
and $F(\pi/2)<F(\beta)< F(0)$ provided that $n=4k+2$ and
$0<\beta<\pi/2$ (cf. Conjecture~\ref{cone}). We do not have a proof of
the previous fact but we include in this paper the following special
cases.

\subsection{The case $m=1$ ($n=2$) and $q=1$}
We have $$F(\beta)=  \frac{\sqrt{2}}{2}\left(3 \sqrt 3
\cos\frac{\beta}{3} + 4 \sin^3\frac{\beta}{3}\right)$$ and
$$F'(\beta)=-\frac{ \sin\frac{\beta}{2} (\sqrt 3 - 2 \sin
\frac{2\beta}{3})}{ 4}.$$ Thus $F'(\beta)=0$ if and only if
$\beta=0$ or $\beta=\frac{\pi}{2}$. The minimum of $F(\beta)$ is
$F(\frac{\pi}{2})=\frac{5\sqrt{2}}{2}$ and the maximum is
$F(0)=\frac{3\sqrt{6}}{2}$.
\subsection{The case $m=2$ ($n=4$) and $q=1$} In this case
\[\begin{split}F(\beta)&=\frac{\sqrt{2}}{8}\left(10 \sqrt{5 + 2 \sqrt{5}} \cos[\beta/5] -
 5 \sqrt{5 - 2 \sqrt 5} \cos[3 \beta/5] + 16 \sin[\beta/5]^5\right).\end{split}\] Then it can be proved that $F$ is
increasing in $[0,\pi/2]$ and $$F(0)=5/4 \sqrt{12.5 + \sqrt
5}\approx4.79845<F(\pi/2)=\sqrt{381/32 + 5 \sqrt 5}\approx4.80485.$$

By differentiating the subintegral expression  \eqref{fbeta} w.r.t
$\beta$ we can easily conclude that $\beta=0$ and $\beta=\pi/2$ are
stationary points of $F$ provided that $q\ge 1$ and $n\in \mathbf
N$. This and some experiments with the software "Mathematica 8"
leads to the following conjecture
\begin{conjecture}\label{cone}
Denote by $[a]$ the integer part of $a$. We conjecture that:
\begin{itemize}
            \item $F_q$ is decreasing on $[0,\frac{\pi}{2}]$ for $q>2$
    \item $F_q$ is nondecreasing (nonincreasing) on $[0,\frac{\pi}{2}]$ for $q\le 2$ and
$\left[\frac{(n+1)q}{2}\right]$ is an even (odd) integer.
\end{itemize}
\end{conjecture}
\section{Appendix~B}
In this appendix we offer some numerical estimation that confirm tha  our conjecture is true, at least for $q=1$.
Whereas as is observed in \eqref{fn} for $0\le \beta\le \pi$
\begin{equation}\label{ff}f(\beta)=f_n(\beta)=\sum_{k=1}^s \sin^s\frac{-\beta+k
\pi}{s}\end{equation} is the constant $2/B(\frac12,\frac{s}2)$ for
even $s$, if $s=n+1$ is odd the maximum exceeds
$2/B(\frac12,\frac{s}2)$ by a tiny amount that is very nearly
$$\frac{4}{\pi} \phantom. \frac1{s+2} \frac2{s+4} \frac3{s+6} \cdots
\frac{s}{3s} = \frac4\pi s! \frac{s!!}{(3s)!!} = (27+o(1))^{-s/2}$$
for large $s$.
 Here and later we use "$u!!$" only
for positive odd $u$ to mean the product of all odd integers in
$[1,u]$; that is, $u!! := u!/(2^v v!)$ where $u=2v+1$. In order to
outline the proof of the last statement we do as follows.

For $s=2m+1$ we define the function $g$ as follows
$$g(x): = f(x+\frac\pi2) = g(-x) = -g(x+s\pi),$$ which has a finite
Fourier expansion in cosines of odd multiples of $X := x/s$, namely
$$f(x) = (-1)^m 2^{-s} \sum_{j=0}^s (-1)^j {s\choose j} \frac{\cos
\phantom. tX}{\sin \frac{\pi t}{2s}}$$ where $t = s-2j$. We deduce
from \eqref{ff} that
$$ f(\beta)-f(\beta+\pi) = 2\phantom.\sin^s (\beta/s),$$ from which
it follows that $g(x)$ is maximized somewhere in $|x| \leq \pi/2$,
but that changing the optimal $x$ by a small integral multiple of
$\pi$ reduces $g$ by a tiny amount; this explains the near-maxima we
observed at $x=\pm\pi$ for $2|m$, and indeed the further
oscillations for both odd and even $m$ that we later noticed as $s$
grows further.

This also suggests that in and near the interval $|x| \leq \pi/2$
our function $g$ should be very nearly approximated for large s by
an even periodic function $\tilde g(x)$ of period $\pi$. We next
outline the derivation of such an approximation, with $\tilde g$
having an explicit cosine-Fourier expansion $$\tilde g(x) = g_0 +
g_1 \cos 2x + g_2 \cos 4x + g_3 \cos 6x + \cdots$$ where $g_0 =
2/B(\frac12,\frac{s}2)$ and, for $l>0$, $$g_l = (-1)^{m+l-1}
\frac4\pi \frac{s!}{2l+1} \frac{((2l-1)s)!!}{((2l+1)s)!!}$$ with the
double-factorial notation defined as above. Thus $$\tilde g(x) = g_0
+ (-1)^m \frac{4s!}\pi \left(\frac{s!!}{(3s)!!} \cos 2x - \frac13
\frac{(3s)!!}{(5s)!!} \cos 4x + \frac15 \frac{(5s)!!}{(7s)!!} \cos
6x \mp \cdots \right).$$ For large $s$, this is maximized at $x=0$
or $x=\pm\pi/2$ according as $m$ is even or odd. Since we already
know by symmetry arguments that $g'(0) = g'(\pm \pi/2) = 0$, this
point or points will also be where g is maximized, once it is
checked that $g - \tilde g$ and its first two derivatives are even
tinier there.

The key to all this is the partial-fraction expansion of the factor
$1 / \sin (\pi t /2s)$ in the Fourier series of $g$, obtained by
substituting $\theta = \pi t / 2s$ into $$\frac1{\sin \pi\theta} =
\frac1\pi \sum_{l=-\infty}^\infty \frac{(-1)^l}{\theta-l}$$ with the
conditionally convergent sum interpreted as a principal value or
Ces\'aro limit etc. On the other hand the main term, for $l=0$,
yields the convolution of $\cos^s (x/s)$ with a symmetrical square
wave, which is thus maximized at $x=0$ and almost constant near
$x=0$; we identify the constant with $2/B(\frac12,\frac{s}2)$ using
the known product formula for
$$\int_{-\pi/2}^{\pi/2} \cos^s X \phantom. dX.$$ The new observation is
that each of the error terms $(-1)^l/(\theta-l)$ likewise yields the
convolution with a square wave of $$(-1)^l \cos(2lx) \phantom.
\cos^s(x/s).$$ If we approximate this square wave with a constant,
we get the formula for $g_l$ displayed above, via the formula for
the $s$-th finite difference of a function $1/(j_0-j)$. The error in
this approximation is still tiny (albeit not necessarily negative)
because $\cos^s (x/s)$ is minuscule when $x$ is within $\pi/2$ of
the square wave's jump at $\pm \pi s / 2$.

We've checked these approximations numerically to high precision
(modern computers and gp make this easy) for $s$ as large as $100$
or so, in both of the odd congruence classes mod $4$, and it all
works as expected; for example, when $s=99$ we have $f(0) - g_0 =
2.57990478176660\ldots \cdot 10^{-70},$ which almost exactly matches
the main term $g_1 = (4/\pi) \phantom. 99! \phantom. 99!!/297!!$ but
exceeds it by $5.9110495\ldots \cdot 10^{-102},$ which is almost
exactly $g_2 = (4/\pi) \phantom. 99! \phantom. 297!!/(3 \cdot
495!!)$ but too large by $7.92129\ldots \cdot 10^{-120}$, which is
almost exactly $g_3 = (4/\pi) \phantom. 99! \phantom. 495!!/(5 \cdot
693!!),$ etc.; and likewise for $s=101$ except that the maximum
occurs at $\beta = \pi/2$ and is approximated by an alternating sum
$g_1 - g_2 + g_3 \ldots$ (actually here this approximation is exact
because $x=0$).

\subsection*{Acknowledgement} We are  thankful to the referee for some corrections and comments, that have improved this paper.

\end{document}